\documentclass[11pt]{amsart}%
\usepackage{amsmath}%
\usepackage{amssymb}%
\usepackage{amscd}%
\usepackage{color}%
\usepackage{mathrsfs}%
\usepackage[all]{xy}%
\def\ol#1{\overline{#1}}% 		overline
\def\wh#1{\widehat{#1}}% 	wide hat
\def\wt#1{\widetilde{#1}}% 	wide tilde
\theoremstyle{plain}
    \newtheorem{theorem}{Theorem}[section]
    
    \newtheorem{proposition}[theorem]{Proposition}
    \newtheorem{lemma}[theorem]{Lemma}
    
\theoremstyle{definition}
    \newtheorem{definition}[theorem]{Definition}
    
    \newtheorem{example}[theorem]{Example}
    \newtheorem{remark}[theorem]{Remark}
\def\Alphabet{A,B,C,D,E,F,G,H,I,J,K,L,M,N,O,P,Q,R,S,T,U,V,W,X,Y,Z}%  Capitalized Alphabet
\def\alphabet{a,b,c,d,e,f,g,h,i,j,k,l,m,n,o,p,q,r,s,t,u,v,w,x,y,z}%	lowercase alphabet
\def\endpiece{xxx}%									marks end of list
\def\makeAlphabet[#1]{\expandafter\makeA#1,xxx,}%		Ex. \makeAlphabet[A,B]
\def\makealphabet[#1]{\expandafter\makea#1,xxx,}%		Ex. \makealphabet[c,d]
\def\makeA#1,{\def\temp{#1}\ifx\temp\endpiece\else%
\mkbb{#1}\mkfrak{#1}\mkbf{#1}\mkcal{#1}\mkscr{#1}\expandafter\makeA\fi}%
\def\makea#1,{\def\temp{#1}\ifx\temp\endpiece\else\mkfrak{#1}\mkbf{#1}\expandafter\makea\fi}%
\def\mkbb#1{\expandafter\def\csname bb#1\endcsname{\mathbb{#1}}}%      Define bb
\def\mkfrak#1{\expandafter\def\csname fr#1\endcsname{\mathfrak{#1}}}%    Define frak
\def\mkbf#1{\expandafter\def\csname b#1\endcsname{\mathbf{#1}}}%           Define bold letters
\def\mkcal#1{\expandafter\def\csname c#1\endcsname{\mathcal{#1}}}%       Define calligraphy
\def\mkscr#1{\expandafter\def\csname s#1\endcsname{\mathscr{#1}}}%       Define script
\def\makeop[#1]{\xmakeop#1,xxx,}%					Ex. \makeop[Hom,Spec]
\def\mkop#1{\expandafter\def\csname #1\endcsname{{\mathrm{#1}}}} % 
\def\xmakeop#1,{\def\temp{#1}\ifx\temp\endpiece\else\mkop{#1}\expandafter\xmakeop\fi}%
\makeAlphabet[\Alphabet]%				Define bb, frak, bf, cal for Capitalized Alphabet
\makealphabet[\alphabet]%  				Define frak and bf for uncapitalized alphabet
\makeop[Alt,End,Ext,Hom,Sym,Tor]% 		Homs
\makeop[CH,Pic,div]% 					Chow groups
\makeop[Ker,Im,Coker,Coim,id,pr]% 		Kernels
\makeop[Spec,Spf,Spm,Isom]% 			Spaces
\makeop[Re,Im]%
\makeop[dR,Nis,crys,rig,syn,tor,Cont]% 		Cohomologies
\makeop[Gal,Res,ab,res]%				Galois groups
\makeop[pol]%							Polylogarithm
\makeop[Var,an,Isoc,univ,Cusp]%
\makeop[Lie,coLie,MIC,DR,Gr,Ab,Cone]
\makeop[Meas]
\makeop[Gl,Sl]
\makeop[Eis,tr,Fil,gr,et]

\def\pair#1{\left\langle #1 \right\rangle}

\def\Isocda{\Isoc^{\kern-0.5mm\dagger}}

\def\Exp#1{\exp\left[#1\right]}%
\def\pair#1{\langle#1\rangle}%
\def\markout#1{}%

\begin{document}
%--- the title ------------------------------------------------------------------------------------------------------------------------
\title{Algebraic theta functions and Eisenstein-Kronecker numbers}
\author{Kenichi Bannai and Shinichi Kobayashi}
\address{Graduate school of Mathematics, Nagoya University, Furo-cho Chikusa-ku, Nagoya, Japan 464-8602}
\date{September 11, 2007}
\begin{abstract}
	In this paper, we give an overview of our previous paper concerning the investigation of 
	the algebraic and $p$-adic properties of Eisenstein-Kronecker numbers using Mumford's
	theory of algebraic theta functions.
\end{abstract}
\thanks{Both authors were supported in part by the JSPS postdoctoral fellowship for research abroad.}
\maketitle

%%%%%%%%%%%%%%%%%%%%%%%%%%%%%%%%%%%%%%%%%%%%%%%%%%%%%
%
%
%
\section{Introduction}
%
%
%
%%%%%%%%%%%%%%%%%%%%%%%%%%%%%%%%%%%%%%%%%%%%%%%%%%%%

In the paper \cite{BK1}, we used Mumford's theory of algebraic theta functions to study the algebraic and $p$-adic 
properties of Eisenstein-Kronecker numbers, which are analogues in the case of imaginary quadratic fields
of the classical generalized Bernoulli numbers, or more precisely, special values of Hurwitz zeta funtions.
The purpose of this paper is to give an overview of the main 
arguments of \cite{BK1}, highlighting the main ideas.  
We will first prove that the generating function for Eisenstein-Kronecker numbers is in fact given by
a theta function which we call the Kronecker theta function.  This theta function differs from the two-variable Jacobi
theta function studied by Zagier \cite{Zag} by a simple exponential factor.
Then we will show how to use Mumford's theory of 
algebraic theta functions to 
study the algebraic and $p$-adic properties of this generating function.   Using this result, when $p$ is an ordinary 
prime, we will
construct a $p$-adic measure interpolating Eisenstein-Kronecker numbers.   We will then use this measure to 
construct the two-variable $p$-adic measure, constructed by Yager \cite{Yag}, interpolating special
values of Hecke $L$-functions.  Our construction may also be used to construct the two-variable $p$-adic measure 
originally constructed by Manin-Vishik \cite{MV} and Katz \cite{Ka}.
We refer the reader to \cite{BK1} for details.  

The detailed content of this paper is as follows.  We first review the cyclotomic case
to demonstrate the type of result we are aiming to prove.
Let $\Gamma := 2 \pi i \bbZ$.   The cyclotomic analogue of the Eisenstein-Kronecker number is the following.

\begin{definition}
	For any $z_0 \in \bbC$ and integer $b > 0$, we let
	$$
		e_b^*(z_0) :=  {\sum_{n \in\Gamma}}^*  \frac{1}{(z_0 + n)^b},
	$$
	where $\sum^*$ means the sum taken over all $n \in \Gamma$ other than $-z_0$ if $z_0 \in \Gamma$.
\end{definition}
The arithmetic importance of $e_b^*(z_0)$ is its relation to critical $L$-values of Dirichlet characters.
Let $\frf = (f)$ be an ideal in $\bbZ$ and suppose $\chi$ is a Dirichlet character of conductor $\frf$ such that 
$\chi(-\alpha) = (-1)^b \chi(\alpha)$.
Then we have by definition
\begin{equation}\label{eq: Dirichlet}
	L_\frf(\chi, b) = \frac{(2 \pi i)^b}{2f^b} \sum_{\alpha \in (\bbZ/\frf)^\times} \chi(\alpha) e_b^*(2 \pi i \alpha/f)
\end{equation}
for $b > 0$.  We define a function $g(z)$ on $\bbC/ \Gamma$ by
$$
	g(z) := \frac{\exp(z)}{\exp(z)-1} - \frac{1}{2}
$$
(we subtract $\frac{1}{2}$ so that $g(z)$ becomes an odd function satisfying $g(-z) = - g(z)$)
and for any $z_0 \in \bbC$, we let $g_{z_0}(z) := g(z + z_0)$.  Then $g_{z_0}(z)$ is
a generating function 
\begin{equation}\label{eq: gen one}
	g_{z_0}(z) = \delta_{z_0} z^{-1} +  \sum_{b > 0} (-1)^{b-1} e_b^*(z_0) z^{b-1}
\end{equation}
of $e_{b}^*(z_0)$ for $b > 0$, where $\delta_x=1$ if $x \in \Gamma$ and is zero otherwise.
We fix an embedding $i : \ol\bbQ \hookrightarrow \bbC$.
If $z_0 \in \Gamma \otimes \bbQ$, then the Laurent expansion of $g_{z_0}(z)$ has coefficients in 
$\ol\bbQ$. Hence we have
$
	e^*_b(z_0) \in \ol \bbQ.
$
The mechanism behind this proof is the following.  $\bbC/\Gamma$ has an algebraic model 
given by the multiplicative group
$\bbG_m$ defined over $\bbQ$, with a uniformization $\xi: \bbC/\Gamma \cong \bbG_m(\bbC)$ 
given by $z \mapsto \exp(z)$.  Then $g(z)$ corresponds to the rational function $t/(t-1) - \frac{1}{2}$ 
on $\bbG_m$ defined over $\bbQ$, and the algebraicity is a consequence of this fact.

Using this generating function, we can also construct a $p$-adic measure interpolating 
the values $e_{b}^*(z_0)$.
Throughout this paper, we fix an embedding
$i_p : \ol\bbQ \hookrightarrow \bbC_p$, and denote by $W$ the
ring of integers of the completion of the maximal unramified extension of $\bbQ_p$ in $\bbC_p$.
If $z_0$ is a torsion point $\not=1$ in $\bbG_m(\bbC)$ of order prime to $p$, then $g_{z_0}(z)$ corresponds through 
$\xi$ to a rational function defined over $W$, and we have
$$
	\wh g_{z_0}(T) := g_{z_0}(z)|_{z = \log(1+T)} \in W[[T]],
$$
where $T$ is the formal parameter of the formal group $\wh \bbG_m$, with relation $T = t-1$.
Using the standard dictionary between $p$-adic measures on $\bbZ_p$ and power series in $W[[T]]$
(see for example \cite{Hid}), we define
the measure $\mu_{z_0}$ to be the measure  on $\bbZ_p$ corresponding to the power series $\wh g_{z_0}(T)$.
Since $\partial_{\log, T} := (1+T) \partial_T = \partial_z$, this measure satisfies the interpolation property
\begin{equation}\label{eq: interpolate one}
	\int_{\bbZ_p} x^{b-1} d \mu_{z_0}(x) = (-1)^{b-1} (b-1)! e_b^*(z_0)
\end{equation}
for $b > 1$.  This measure may be used to construct the $p$-adic $L$-function interpolating 
critical $L$-values of Dirichlet characters.

The goal of \cite{BK1}, and also of this paper, is to use similar ideas to study Eisenstein-Kronecker numbers.  
Let $K$ be an imaginary quadratic  field with ring of integers $\cO_K$.  We assume in this paper for simplicity 
that the class number of $K$ is one.   See \cite{BK1} for the general case.
Fix an ideal $\frf \subset \cO_K$.  Then there exists by the theory of 
complex multiplication some complex period 
$\Omega \in \bbC$ such that for the lattice $\Gamma: = \Omega \frf$, the complex torus $\bbC/ \Gamma$ has an 
algebraic model $E$ defined over $K$, with standard uniformization $\bbC/\Gamma \cong E(\bbC)$.
Let $A$ be the area of the fundamental domain of $\Gamma$ divided by $\pi = 3.1415\cdots$, 
and $\pair{z,w} = \exp((z \ol w - w \ol z) /A)$, where $\ol z$ and $\ol w$ are the complex conjugates of $z$ and $w$.

\begin{definition}
	For $z_0$, $w_0 \in \bbC$ and integers $b > a+2$,  we define the Eisenstein-Kronecker number
	$e^*_{a,b}(z_0, w_0)$ to be the sum
	$$
		e^*_{a,b}(z_0, w_0) := {\sum_{\gamma \in \Gamma}}^* \frac{(\ol z_0 + \ol\gamma)^a}{(z_0 + \gamma)^b}
		\pair{\gamma, w_0},
	$$	
	where ${\sum}^*$ means the sum taken over all $\gamma \in \Gamma$ other than $ -z_0$ if $z_0 \in \Gamma$.
	We may extend the definition to $a \geq 0$, $b > 0$ using analytic continuation (See Definition \ref{def: EK}).
\end{definition}

The arithmetic importance of $e^*_{a,b}(z_0,w_0)$ is its relation to critical $L$-values of Hecke characters on $K$
(see \eqref{eq: hecke}.)   The observation for our research is the fact that the 
generating function, which plays the role of $g(z)$ above, is given by the Kronecker theta function $\Theta(z,w)$.
Since this function is not a rational function,  the translation $\Theta_{z_0, w_0}(z,w)$ of this function
is taken using the theory of algebraic theta functions of Mumford, which uses extra exponential factors
to preserve the algebraicity.  As an analogy of \eqref{eq: gen one}, we have
\begin{multline}\label{eq: gen two}
	\Theta_{z_0, w_0}(z,w) = \pair{w_0, z_0}  \delta_{z_0} z^{-1} +  \delta_{w_0} w^{-1}\\
	+ \sum_{a, b\geq 0} (-1)^{a+b} \frac{e_{a,b+1}^*(z_0, w_0)}{A^a a!} z^{b} w^{a}.
\end{multline}
We will  systematically use Mumford's theory to study the algebraic and $p$-adic properties of this function.
In particular, we prove that $e^*_{a,b}(z_0, w_0)$ is algebraic if $z_0$, $w_0 \in \Gamma\otimes\bbQ$.
In addition, when $p \geq 5$ is ordinary (i.e. if $p$ splits as $p = \frp\frp^*$ in $\cO_K$), 
we will prove a $p$-integrality statement
for $\Theta_{z_0,w_0}(z,w)$.   We will then use this result to construct a $p$-adic measure analogous to 
\eqref{eq: interpolate one} interpolating the values $e^*_{a,b}(z_0, w_0)$.  
At the end, we will use this measure to construct Yager's measure interpolating critical values of
Hecke $L$-functions.

Construction of a two-variable distribution when $p$ is supersingular interpolating Eisenstein-Kronecker numbers 
will be given in a subsequent paper \cite{BK2}.  The construction of the two-variable $p$-adic distribution 
as in \cite{BK1} and \cite{BK2} play an important role in explicit calculation of the $p$-adic elliptic polylogarithm in \cite{BKT}, 
extending previous results of \cite{Ba} to two variables.\\

The authors would like to thank the organizers Ki-ichiro Hashimoto, Yukiyoshi Nakkajima and  Hiroshi Tsunogai for
the opportunity to give a talk at this conference.  The authors would also like to thank the referee for carefully reading the
manuscript and giving appropriate comments.  Part of this research was conducted while the first author was visiting the 
\'Ecole Normale Sup\'erieure at Paris, and the second author Institut de Math\'ematiques de Jussieu.  The authors would like to thank their
hosts Yves Andr\'e and Pierre Colmez for hospitality. 

%%%%%%%%%%%%%%%%%%%%%%%%%%%%%%%%%%%%%%%%%%%%%%%%%%%%%
%
%
%
\section{Kronecker theta function}
%
%
%
%%%%%%%%%%%%%%%%%%%%%%%%%%%%%%%%%%%%%%%%%%%%%%%%%%%%

In this section, we first define and investigate the properties of Eisenstein-Kronecker-Lerch series and 
Eisenstein-Kronecker numbers.  Then, after reviewing the theory of line bundles and theta functions on a general complex
torus, we prove that the Kronecker theta function $\Theta(z,w)$ is a generating function 
of Eisenstein-Kronecker numbers.  In this section, we let $\Gamma$ be a general lattice $\bbC$.

\begin{definition}(\cite{W} VIII \S 12)\label{def: EKL} 
 Let $a$ be an integer $\geq 0$.  
 For $z$, $w \in \bbC$,  we define the 
  \textit{Eisenstein-Kronecker-Lerch function} $K_a(z,w,s)$ by
  \begin{equation} \label{equation: definition of Eisenstein-Kronecker}
     K_a(z,w,s; \Gamma) = {\sum}^*_{\gamma \in \Gamma} 
    \frac{(\ol{z} + \ol{\gamma})^a}{|z + \gamma|^{2s}} \pair{\gamma, w} \quad    (\Re\,s > a/2 + 1),
  \end{equation}
   where $\sum^*$  means the sum taken over all $\gamma \in \Gamma$ other than $-z$
   if $z$ is in $\Gamma$.
  %The series on the right  converges absolutely for $\Re\,s > a/2+1$.
 \end{definition}

We omit $\Gamma$ from the notation if there is no fear of confusion.
The function $K_a(z, w,s)$ may be continued  analytically in $s$ as follows.

\begin{proposition}\label{prop: property K}
	Let $a$ be an integer $\geq 0$.
	The function $K_a(z,w,s)$ continues meromorphically to a function on the whole $s$-plane, 
	with possible poles only at  $s=0$  $($if $a=0$, $z \in \Gamma${}$)$ and at 
     	$s=1$ $($if $a=0$, $w \in \Gamma${}$)$.
\end{proposition}	

\begin{proof}
	See \cite{W} VIII \S13 for the proof.
\end{proof}

From the definition of Eisenstein-Kronecker-Lerch series, we have the following differential equations.

\begin{lemma} \label{lem: diff}
   	Let $a$ be an integer $>0$.  For $z$, $w \in (\mathbb{C} \setminus \Gamma)$,
   	the function $K_a(z,w,s)$ as a real analytic function in $z$ and $w$ satisfies the differential equations
   	\begin{align*} 
         		\partial_{z} K_a(z,w,s) &= -s K_{a+1}(z,w,s+1) \\
           	\partial_{\ol{z}}K_a(z,w,s) &= (a-s)K_{a-1}(z,w,s) \\
           	\partial_{w} K_a(z,w,s) &= (\ol{z} K_{a}(z,w,s) - K_{a+1}(z,w,s) ) /A\\
           	\partial_{\ol{w}}K_a(z,w,s) &= (K_{a-1}(z,w,s-1) - z K_{a}(z,w,s) ) /A.
	\end{align*}
\end{lemma}

Eisenstein-Kronecker numbers are defined as follows.

\begin{definition}[Eisenstein-Kronecker number]\label{def: EK}
	For any $z_0$, $w_0 \in \bbC$ and integers $a \geq 0$, $b > 0$, we define the Eisenstein-Kronecker number
	 $e^*_{a,b}(z_0, w_0)$ by the formula
	$$
		e^*_{a,b}(z_0, w_0) := K_{a+b}(z_0, w_0, b).
	$$
\end{definition}

We next  review the theory of line bundles and theta functions on a general complex torus following Mumford \cite{Mum2}.
Let $V$ be a complex vector space.    For a lattice $\Lambda \subset V$, consider the complex torus 
$\bbT = V/ \Lambda$.   The holomorphic line bundles on  $\bbT$ are classified by the following theorem.

\begin{theorem}[Appell and Humbert]
	The group $\Pic(\bbT)$ of isomorphism classes of holomorphic line bundles on $\bbT$ is isomorphic to the 
	group of pairs $(H, \alpha)$, where
	\begin{enumerate}
		\renewcommand{\theenumi}{\roman{enumi}}
		\renewcommand{\labelenumi}{(\theenumi)}
		\item $H$ is a Hermitian form on $V$.
		\item $E = \Im H$ takes integral values on $\Lambda$.
		\item $\alpha : \Lambda \rightarrow U(1) : = \{ z \in \mathbb{C} \mid |z|=1 \}$ is a map
		such that $\alpha(\gamma+ \gamma') = \exp(\pi i E(\gamma, \gamma')) \alpha(\gamma) \alpha(\gamma')$.
	\end{enumerate}	
\end{theorem}
One may construct from a pair $(H, \alpha)$ a line bundle $\sL(H, \alpha)$ on $\bbT = V/\Lambda$.
Any meromorphic section of $\sL(H, \alpha)$ over $\bbT$ is given by a meromorphic function 
$\vartheta : V \rightarrow \bbC$ satisfying the transformation formula
\begin{equation}\label{eq: transform}
	\vartheta(v + \gamma) = \alpha(\gamma) \Exp{ \pi H(v,\gamma) + \frac{\pi}{2} H(\gamma, \gamma) } \vartheta(v).
\end{equation}

\begin{definition}
	We call any meromorphic function $\vartheta : V \rightarrow \bbC$ satisfying the transformation formula
	\eqref{eq: transform} a \textit{reduced theta function} associated to $\sL(H, \alpha)$. 
	%The term \textit{normalized theta function} is used in some literature.
\end{definition}

\begin{example}
	We define $\theta(z)$ to be a reduced theta function associated to $\sL(H,\alpha)$ for $V = \bbC$,
	\begin{align*}
		H(z_1, z_2)& = \frac{z_1 \ol z_2}{\pi A},  &   \alpha(\gamma) &= \begin{cases} 1  & \gamma \in 2 \Gamma \\
		-1  & \text{otherwise}. 
		\end{cases}
	\end{align*}
	The divisor of $\theta(z)$ is  $(0)$. We normalize this function so that $\theta'(0) = 1$.
\end{example}

The function $\theta(z)$ is the Weierstrass $\sigma$-function $\sigma(z)$ up to a simple exponential factor.
This function was used by Roberts \cite{Rob1} in his construction of elliptic units.

\begin{example}[Kronecker theta function]
	We define $\Theta(z,w)$ to be the Kronecker theta function
	$$
		\Theta(z,w) := \theta(z+w)/\theta(z)\theta(w).
	$$ 
	This function is a reduced theta function associated to the line bundle $\sL(H, \alpha)$ for
	$V = \bbC \times \bbC$,
	\begin{align*}
		H((z_1, w_1), (z_2, w_2))& = \frac{z_1 \ol w_2 + z_2 \ol w_1}{\pi A},  \\  
		 \alpha(\gamma_1, \gamma_2) &= \frac{\gamma_1 \ol \gamma_2 - \gamma_2 \ol \gamma_1}{2A}.
	\end{align*}
	This line bundle is the Poincar\'e bundle.
	The divisor of $\Theta(z,w)$ is $\Delta - (E \times \{ 0 \}) - (\{0 \} \times E)$ of $E^2 = (\bbC/\Gamma)^2$, 
	where $\Delta$ is the  image of the map $x \mapsto (x,-x)$, and the residue at $z=0$ and $w=0$ is the constant
	\textit{one}. This function satisfies the translation formula
	\begin{equation}\label{eq: Theta}
		\Theta(z+ \gamma_1, w + \gamma_2) = \Exp{\frac{\gamma_1 \ol \gamma_2}{A}} 
		\Exp{ \frac{z \ol \gamma_2+ w \ol \gamma_1}{A} } \Theta(z,w)
	\end{equation}
	for any $\gamma_1$, $\gamma_2 \in \Gamma$.
\end{example}

We will next give a relation between the Kronecker theta function and Eisenstein-Kronecker-Lerch series.

\begin{lemma}\label{lem: K}
	Let $f(z,w) := \Exp{z \ol w/A} K_1(z,w,1)$.  Then this function satisfies the following.
	\begin{enumerate}
		\item $f(z,w)$ satisfies the transformation formula \eqref{eq: Theta}.
		\item $f(z,w)$ is a meromorphic function in $z$ and $w$, holomorphic except simple poles when 
			$z \in \Gamma$ or $w \in \Gamma$.
		\item The residue of $f(z,w)$ at $z=0$ and $w=0$ is equal to one.
	\end{enumerate}
\end{lemma}

\begin{proof}
	The proof is given in \cite{BK1} Proposition 1.12 (ii).
\end{proof}

\begin{theorem}[Kronecker]
	$$
		\Theta(z,w) = \Exp{ \frac{z \ol w}{A}} K_1(z,w,1).
	$$
\end{theorem}

\begin{proof}
	By the property of $\Theta(z,w)$ and Lemma \ref{lem: K}, the difference 
	$\Theta(z,w) - \Exp{ z \ol w/A} K_1(z,w,1)$
	is a holomorphic function on $\bbC^2$ satisfying \eqref{eq: Theta}.  Since $H$ is not positive definite,
	the line bundle $\sL(H, \alpha)$ is not ample hence has no non-zero holomorphic sections.  
	Therefore, the above difference must be zero as desired.
\end{proof}

The above equality is used to prove that $\Theta_{z_0,w_0}(z,w)$ below is a generating function
for the values $e_{a,b}^*(z_0,w_0)$.

\begin{theorem} For any $z_0$, $w_0 \in \bbC$, let
	$$
		\Theta_{z_0, w_0}(z,w) :=  \Exp{- \frac{z_0 \ol w_0}{A}} \Exp{ - \frac{z \ol w_0 + w \ol z_0}{A} } 
		\Theta(z + z_0,w + w_0).
	$$
	Then the Laurent expansion of this function at the origin is given by \eqref{eq: gen two}.
\end{theorem}

\begin{proof}
	Let $\wt K_{a+b}(z,w,b) := \Exp{-w \ol z/A} K_{a+b}(z,w,b)$.  Then by Lemma \ref{lem: diff}, we have
	\begin{align*}
		\partial_z \wt K_{a+b}(z,w,b) &= -b \wt K_{a+b+1}(z,w,b+1),   \\
		\partial_w \wt K_{a+b}(z,w,b) &= - \wt K_{a+b+1}(z,w,b)/A.	
	\end{align*}
	Hence when $z_0$, $w_0 \not\in \Gamma$, the coefficient of $z$ and $w$ in the Taylor expansion 
	of $\wt K_{1}(z + z_0, w+ w_0,1)$ at the origin is given by
	$$
		\sum_{a,b  \geq0} (-1)^{a+b} \frac{\wt K_{a+b+1}(z_0, w_0, b+1)}{a! A^a} z^{b} w^a.
	$$
	By definition and the previous theorem, $\Theta_{z_0, w_0}(z,w)$ is equal to
	$$
		 \Exp{\frac{w_0 \ol z_0}{A}} \Exp{\frac{(z + z_0) \ol w + (w + w_0) \ol z}{A}}\wt K_1(z + z_0,w + w_0,1).
	$$	
	Since $\Theta_{z_0, w_0}(z,w)$ is holomorphic at the origin and the second exponential above is equal to 
	\textit{one} when $\ol z = \ol w = 0$, our assertion follows from the fact that 
	$e^*_{a,b}(z_0, w_0) = \Exp{w_0 \ol z_0/A} \wt K_{a+b}(z_0, w_0, b)$.  The case when $z_0$ or $w_0 \in \Gamma$
	follows using a similar argument, paying careful attention to the poles of $\Theta_{z_0,w_0}(z,w)$.
	See \cite{BK1} \S 1.4 for details.
\end{proof}

%%%%%%%%%%%%%%%%%%%%%%%%%%%%%%%%%%%%%%%%%%%%%%%%%%%%%
%
%
%
\section{Algebraicity}
%
%
%
%%%%%%%%%%%%%%%%%%%%%%%%%%%%%%%%%%%%%%%%%%%%%%%%%%%%

In this section, using the fact that $\Theta(z,w)$ is a generating function for Eisenstein-Kronecker numbers,
we give a simple proof of Damerell's theorem concerning the algebraicity of such numbers.
Let $K$ be an imaginary quadratic field of class number one as in the introduction, and let $\frf$ be an ideal in 
$\cO_K$.  Then by the theory of complex multiplication, there exists an elliptic curve 
$E$ and an invariant differential $\omega$ defined over $K$ and a complex number 
$\Omega$, such that we have a complex uniformization
$$
	\xi : \bbC/ \Gamma \xrightarrow\cong E(\bbC)
$$
where $\Gamma := \Omega \frf$ corresponds to the period lattice of $\omega$.
We first prove that the Laurent expansion of $\Theta(z,w)$ at the origin
has coefficients in $K$.

\begin{lemma}\label{lem: alg origin}
	The Laurent expansion of $\Theta(z,w)$ at the origin has coefficients in $K$.
\end{lemma}

\begin{proof}
	We first prove that the Taylor expansion of $\theta(z)$ at the origin has coefficients in $K$.  Take an $\alpha \in \cO_K$
	such that $\alpha \not\in \bbZ$.  The function $\theta_\alpha(z) := \theta(z)^{N(\alpha)}/\theta(\alpha z)$ is periodic 
	with respect to $\Gamma$, and in fact corresponds to a rational function of $E$ defined over $K$.  
	Hence the Taylor expansion
	of $\theta_\alpha(z)$ at $z=0$ has coefficients in $K$.  The statement for $\theta(z)$ is obtained from
	the statement for $\theta_\alpha(z)$ by inductively comparing the coefficients.
	The statement for $\Theta(z,w)$ now follows, since $\Theta(z,w) := \theta(z+w)/\theta(z)\theta(w)$.
\end{proof}

\begin{remark}
	The condition $\alpha \not\in \bbZ$ is necessary to inductively compare the coefficients of $\theta(z)$
	and $\theta_\alpha(z)$.  Hence the assumption that $\Gamma$ has an $\cO_K$-structure plays
	an important role in the proof.
\end{remark}

In order to prove the algebraicity of the Laurent expansion of $\Theta_{z_0, w_0}(z,w)$,  we will use the 
theory of algebraic theta functions due to Mumford.  We first review the general theory.
Consider again a general complex torus $\bbT = V / \Lambda$.  
Assume that $\bbT$ has an algebraic model $A$ defined over a number field $F$, 
with an isomorphism $\bbT \xrightarrow\cong A(\bbC)$.
Let $\vartheta(v)$ be a theta function associated to $\sL$, 
satisfying the translation formula \eqref{eq: transform}.  
We assume in addition that 
\begin{enumerate}
	\item $\vartheta(v)$ is odd, in other words $\vartheta(-v) = - \vartheta(v)$.
	\item The divisor of $\vartheta(v)$ (the theta divisor) is defined over $F$.
\end{enumerate}
Suppose we know the property of $\vartheta(v)$ at a neighborhood of $v  = 0$.  Mumford's theory allows us
to deduce the properties of $\vartheta(v)$ at arbitrary torsion points from its property at $v=0$.

We take a $v_0 \in \bbQ \otimes \Lambda$ corresponding to a torsion point of $A$.  
We let $n$ be an integer such that $n v_0 \in 2 \Lambda$, and we let $q_0$ be a point
such that $2 q_0 = v_0$.  Then we have the following:

\begin{lemma}
	The function
	$$
		\rho(v) := \Exp{ - \pi H(n v, q_0)} \vartheta(q_0 + nv) \vartheta(nv)^{-1}
	$$
	is meromorphic in $v$ and is periodic with respect to $\Lambda$.
\end{lemma}

\begin{proof}
	Since $H$ is a hermitian form, the exponential term is holomorphic in $v$.
	The lemma follows from the transformation formula \eqref{eq: transform}, noting that
	$n q_0 \in \Lambda$ for our choice of $n$.
\end{proof}

We have assumed that the divisor of $\vartheta(v)$ is defined over $F$.  Then from the explicit description of $\rho(v)$,
we see that the divisor of $\rho(v)$ is defined over $\ol F$.  
This implies that there exists some constant 
$%\begin{equation}\label{equation: c}
	c \in \bbC^\times
 $%\end{equation}
such that the function $c \rho(v)$ corresponds to
a rational function on $A$ defined over $\ol F$. Since $\rho(-v)^{-1} \rho(v) = (c\rho(-v))^{-1} (c\rho(v)) $,
we see that
\begin{align}\label{equation: rho}
	\rho(-v)^{-1} \rho(v)% &= \Exp{- 2 \pi H(n v, q_0)} \frac{\vartheta(q_0 + nv ) \vartheta(-nv)}{\vartheta(q_0 -nv) \vartheta(nv)} \\
	&=\Exp{- 2 \pi H(n v, q_0)} \frac{\vartheta(q_0 + nv)}{\vartheta(-q_0 + nv)} 
\end{align}
is also a rational function defined over $\ol F$ independent of the choice of $c$.  
Taking change of coordinates $v \mapsto v + (q_0/n)$, we have
\begin{align}\label{equation: r}
	r(v) &:= \rho \left(- v - (q_0/n) \right)^{-1} \rho \left(v+ (q_0/n) \right) \\
	&= \Exp{- \pi H(n v, v_0) - \frac{\pi}{2} H(v_0, v_0) } \vartheta(v_0 + nv) \vartheta(nv)^{-1},\nonumber
\end{align}
which corresponds to a rational function defined over $\ol F$ independent of the choice of $c$.
This method of first taking $q_0$ in order to remove the constant $c$ is the key in Mumford's theory.
We define Mumford's algebraic theta function (at point $v_0$) by 
\begin{equation}\label{equation: Mumford theta}
	\vartheta_{v_0}^{\cM} (v) := r(v/n) \vartheta(v).
\end{equation}
Then we have
\begin{equation*}
	\vartheta_{v_0}^{\cM} (v)  =  \Exp{- \pi H(v, v_0) - \frac{\pi}{2} H(v_0, v_0) } \vartheta(v+v_0).
\end{equation*}
Since $r(v)$ corresponds to a rational function defined over $\ol F$, one may deduce
algebraicity results at $v=0$  for $\vartheta^\cM_{v_0}(v)$ from corresponding results for $\vartheta(v)$.\\

Since $\Theta(z,w)$ is odd and its divisor $\Delta - (E \times \{ 0 \}) - (\{0 \} \times E)$ is defined over $K$,
we may apply the above theory. Let $z_0$, $w_0 \in \Gamma \otimes \bbQ$.  Then by definition, we have
$$
	\Theta_{z_0, w_0}(z,w) = \Exp{\frac{w_0 \ol z_0-z_0 \ol w_0}{2A}} \Theta_{z_0, w_0}^\cM(z,w).
$$
This gives the following theorem.

\begin{theorem}
	Suppose $z_0$, $w_0 \in \Gamma\otimes\bbQ$.  Then the Laurent expansion of $\Theta_{z_0, w_0}(z,w)$ 
	at the origin has coefficients in $\ol K$.
\end{theorem}

\begin{proof}
	Since the difference between $\Theta_{z_0,w_0}$ and $\Theta^\cM_{z_0,w_0}$ is a root of unity, it is sufficient
	to prove the statement for $\Theta^\cM_{z_0,w_0}$.
	In the construction of $\Theta_{z_0, w_0}^\cM(z,w)$, we may take $\rho(z,w)$ hence $r(z,w)$ 
	to be a rational function on $E \times E$ defined over $\ol K$.  
	Hence the Laurent expansion of $r(z,w)$ at the origin has coefficients in $\ol K$.
	Since 
	$
		\Theta_{z_0, w_0}^\cM(z,w):= r(z/n, w/n) \Theta(z,w),
	$
	our assertion now follows from Lemma \ref{lem: alg origin}.
\end{proof}

Since $\Theta_{z_0,w_0}(z,w)$ is a generating function for Eisenstein-Kronecker numbers, 
the above theorem gives the following.

\begin{theorem}[Damerell]
	Let $\Gamma$ be as above.  For any  $z_0$, $w_0 \in \Gamma\otimes\bbQ$ and integers $a>0$, $b \geq 0$,
	we have $$e_{a,b}^*(z_0, w_0)/A^a \in \ol K.$$
\end{theorem}

%%%%%%%%%%%%%%%%%%%%%%%%%%%%%%%%%%%%%%%%%%%%%%%%%%%%%
%
%
%
\section{$p$-integrality}
%
%
%
%%%%%%%%%%%%%%%%%%%%%%%%%%%%%%%%%%%%%%%%%%%%%%%%%%%%

In this section, we prove the $p$-integrality of the Laurent expansion of $\Theta_{z_0,w_0}(z,w)$
with respect to the variable of the formal parameter, when the elliptic curve $E$ has good ordinary
reduction at $p$.  We will then use this fact to construct $p$-adic measures interpolating 
Eisenstein-Kronecker numbers.

Let $p \geq 5$ be a prime for which $E$ has good ordinary reduction.  
This condition is equivalent to saying that $(p)$ splits as $(p) = \frp\frp^*$ in $\cO_K$.
We fix a Weierstrass model $\cE$ of $E$ over $\cO_{K}$, given by the equation
 $$
 	\cE: y^2 = 4 x^3 - g_2 x - g_3
 $$
with good reduction at $p$.  We denote by $\wh\cE$ the formal group of $\cE$ with respect to the 
parameter $t= -2x/y$, and we denote by $\lambda(t)$ the formal logarithm of $\wh\cE$.
We also fix an embedding $i_p : \ol\bbQ \hookrightarrow \bbC_p$ such that the completion
of $K$ in $\bbC_p$ is $K_\frp$, and let $W$ be the
ring of integers of the completion of the maximal unramified extension of $\bbQ_p$ as in the introduction.
Then we have the following lemma, due to Bernardi, Goldstein and Stephens.

\begin{lemma}[\cite{BGS} Proposition III.1]\label{lem: BGS}
	We let
	$
		\wh \theta(t) := \theta(z)|_{z= \lambda(t)}
	$
	be the formal composition of the Taylor expansion of $\theta(z)$ at $z=0$ with $\lambda(t)$.
	Then we have
	$$
		t^{-1} \wh\theta(t)  \in W[[t]]^\times.
	$$
\end{lemma}

Using this lemma, we have the following.

\begin{proposition}
	Let $\wh\Theta(s,t) := \Theta(z,w)|_{z = \lambda(s), w=\lambda(t)}$ be the formal composition of the 
	Laurent expansion of $\Theta(z,w)$ at the origin with $z= \lambda(s)$ and $w= \lambda(t)$.  Then we have
	$$
		\wh\Theta(s,t) - s^{-1} - t^{-1} \in W[[s,t]].
	$$
\end{proposition}

\begin{proof}
	The statement follows from  the previous lemma and the definition of $\Theta(z,w)$.
\end{proof}

Again using the method of Mumford, we may prove a $p$-integrality statement for 
$\Theta_{z_0, w_0}(z,w)$.

\begin{theorem}
	Let $z_0$, $w_0 \in \Gamma \otimes \bbQ$ be torsion points whose order $n$ is prime to $p$, and let
	$\wh\Theta_{z_0,w_0}(s,t) := \Theta_{z_0,w_0}(z,w)|_{z = \lambda(s), w=\lambda(t)}$ be the formal 
	composition of the Laurent expansion of $\Theta_{z_0,w_0}(z,w)$ at the origin with $z= \lambda(s)$ and $w= \lambda(t)$.
	Then we have
	$$
		\wh\Theta_{z_0,w_0}^*(s,t) :=  
		\wh\Theta_{z_0,w_0}(s,t)  - \pair{w_0, z_0} \delta_{z_0} s^{-1} - \delta_{w_0} w^{-1} 
		\in W[[s,t]].
	$$
\end{theorem}

\begin{proof}
	Since the difference between $\Theta_{z_0,w_0}$ and $\Theta_{z_0,w_0}^\cM$ is
	a root of unity of order prime to $p$, it is sufficient to prove the statement for $\Theta_{z_0,w_0}^\cM$.
	In the construction of $\Theta_{z_0,w_0}^\cM$, we may take $\rho(z,w)$ hence $r(z,w)$ so that it
	is %expressed as quotients of holomorphic functions on $\cE$ which are prime to $p$.
	defined over $W$.   Using the fact that the only possible poles of $r(z,w)$ at the formal neighborhood 
	$\wh\cE \times \wh\cE$ of the origin is the origin itself, we can prove that 
	$$
		\wh r(s,t) := r(z,w)|_{z= \lambda(s), w=\lambda(t)} \in W[s^{-1}, t^{-1}][[s,t]].
	$$  
	By substituting into the equality
	$\wh\Theta_{z_0,w_0}^\cM([n]s,[n]t) = \wh r(s, t) \wh\Theta([n]s,[n]t)$
	the inverse power
	series of $[n]s$ and $[n]t$, which also have coefficients in $W$, we have 
	$\wh\Theta_{z_0,w_0}(s,t) \in W[s^{-1}, t^{-1}][[s,t]]$.
	Our assertion now follows from the explicit shape of the poles of $\Theta_{z_0,w_0}(z,w)$.
\end{proof}

\begin{remark}
	In the above theorem, the condition that $p$ is ordinary is crucial in proving the integrality.
	The coefficients of $\wh\Theta_{z_0,w_0}(s,t)$ is $p$-adically unbounded
	when $p$ is supersingular
\end{remark}

We use the above theorem to construct our $p$-adic measure.  It is known that there exists an
isomorphism of formal groups 
$$
	\eta_p : \wh\cE\xrightarrow\cong\wh\bbG_m
$$
over $W$, which is expressed in the form $\eta_p(t) = \exp( \lambda(t)/\Omega_\frp) -1$
for some suitable $p$-adic period $\Omega_\frp \in W^\times$.
We let $\iota(T) = \Omega_\frp T + \cdots$ be the inverse power series of $\eta_p(t)$, and we let
$$
	\wh\Theta_{z_0,w_0}^{*\iota}(S,T) := \wh\Theta_{z_0,w_0}^*(s,t) |_{s = \iota(S), t = \iota(T)}.
$$
Using the standard dictionary between $p$-adic measures on $(\bbZ_p)^2$ and
formal power series in $W[[S,T]]$, we may define the measure $\mu_{z_0,w_0}$ as follows.

\begin{definition}
	Let $z_0$, $w_0 \in \Gamma \otimes \bbQ$ be torsion points of order prime to $p$.
	We define $\mu_{z_0, w_0}$ to be the $p$-adic measure on 
	$\bbZ_p \times \bbZ_p$ characterized by the formula
	$$
		\int_{\bbZ_p^2} (1+S)^x (1+T)^y d \mu_{z_0, w_0}(x,y) = \wh\Theta_{z_0,w_0}^{*\iota}(S,T).	
	$$
 \end{definition}

 When $z_0$, $w_0 \not\in \Gamma$, we have $\wh\Theta^*_{z_0,w_0} = \wh\Theta_{z_0,w_0}$.
 Note that the differential $\partial_{\log, S} := (1+S) \partial_S$ corresponds to the differential $\Omega_{\frp}^{-1} \partial_z$
 through the equality $z = \lambda \circ \iota(S)$, and similarly for $\partial_{\log,T}$.
 Since $\Theta_{z_0,w_0}(z,w)$ is a generating function of Eisenstein-Kronecker numbers, the measure 
 $\mu_{z_0,w_0}$ in this case satisfies the interpolation property
 \begin{equation}\label{eq: interpolate}
 	\frac{1}{\Omega_\frp^{a+b-1}}
	\int_{\bbZ_p^2} x^{b-1} y^a d\mu_{z_0,w_0}(x,y) =
	(-1)^{a+b-1} (b-1)! \frac{e^*_{a,b}(z_0,w_0)}{A^a}
 \end{equation}
 for integers $a \geq 0$ and $b>0$.  When $z_0$ or $w_0 \in \Gamma$,  we cannot calculate directly
 the interpolation property of $\mu_{z_0,w_0}$.  We calculate instead the interpolation property
 of the restriction of $\mu_{z_0,w_0}$ to $(\bbZ_p^\times)^2$ (See \cite{BK1} Proposition 3.6 for details.)
 
 \begin{remark}
 	In \cite{BK2}, extending previous results of Boxall \cite{Box2} and Schneider-Teitelbaum \cite{ST},
	we develop a theory associating $p$-adic distributions to two-variable power series which are not necessarily
	$p$-integral,  enabling our construction also in the supersingular case.
 \end{remark}

 We use the measure $\mu_{z_0,w_0}$ to construct the two-variable $p$-adic $L$-function of Yager \cite{Yag}, interpolating
 special values of Hecke $L$-functions.  We first review the relation between Eisenstein-Kronecker numbers and
 Hecke $L$-functions.
 
 Let $\varphi$ be a Hecke character of $K$ of infinity type $(1,0)$ with values in $K$.
 Assume in addition that the conductor $\frf$ of $\varphi$ is prime to $p$,
 and assume $w_\frf =1$, where  $w_\frf$ is the number of roots of unity in $K$ congruent to one modulo $\frf$.
 Then there exists a CM elliptic curve $E$ over $K$ whose Gr\"ossencharakter is $\varphi$.  We fix a 
 Weierstrass model of $E$ over $\cO_K$ with good reduction at $p$,  and let $\Gamma$ be its period lattice. 
Let $\Omega$ be a complex number such that $\Gamma = \Omega \frf$.
The Hecke $L$-function $L_\frf(\ol\varphi^a,s)$ for $a \geq 0$
may be expressed in terms of Eisenstein-Kronecker-Lerch series as follows.
\begin{multline*}
	L_\frf(\ol\varphi^a, s) = \sum_{\fra} \frac{\ol\varphi^a(\fra)}{N(\fra)^s}  
	= \sum_{\alpha \in\cO_K} \frac{\ol\varphi^a(\alpha)}{|\alpha|^{2s}}
	=\sum_{\alpha \in(\cO_K/\frf)^\times} K_{a}( \varphi(\alpha), 0,  s; \frf) \\
	= \frac{|\Omega|^{2s}}{\ol\Omega^a} \sum_{\alpha \in(\cO_K/\frf)^\times} K_{a}( \varphi(\alpha) \Omega, 0,  s).
\end{multline*}
Since $A:= A(\Gamma) = N(\frf) \Omega \ol \Omega A(\cO_K) = N(\frf) \Omega\ol\Omega(\sqrt{d_K}/2\pi)$,
where $-d_K$ is the discriminant of $K$, the above equality gives
\begin{equation}\label{eq: hecke}
	\left[ \frac{2 \pi}{\sqrt{d_K}} \right]^a \frac{L_\frf(\ol\varphi^{a+b}, b)}{\Omega^{a+b}}
	= N(\frf)^a \sum_{\alpha \in (\cO_K/\frf)^\times} \frac{e^*_{a,b}(\varphi(\alpha) \Omega, 0)}{A^a}.
\end{equation}
Hence the special values of Hecke $L$-functions may be expressed in terms of Eisenstein-Kronecker numbers.
We now define the measure $\mu_\varphi$ as follows.
 
 \begin{definition}
 	We define the measure $\mu_\varphi$ on $\bbZ_p \times \bbZ_p$ as
	$$
		\mu_\varphi(x,y) =  \sum_{\alpha \in (\cO_K/\frf)^\times} \mu_{\varphi(\alpha)\Omega,0}(x, N(\frf) y).
	$$
 \end{definition}
 
 Then this measure satisfies the interpolation property of Yager.
 
 \begin{proposition}
 	For any integer $a \geq 0$ and $b>0$, we have
	\begin{multline*}
		\frac{1}{\Omega_\frp^{a+b}} \int_{(\bbZ_p^\times)^2} x^{b-1} y^a d \mu_\varphi(x,y) \\
		=(-1)^{a+b-1}(b-1)! \left[ \frac{2 \pi}{\sqrt{d_K}}\right]^a
		\left( 1 - \frac{\varphi(\frp)^{a+b}}{N\frp^{a+1}}\right)
		\left( 1 - \frac{\ol\varphi(\ol\frp)^{a+b}}{N\ol\frp^{b}}\right)
		\frac{L_\frf(\ol\varphi^{a+b},b)}{\Omega^{a+b}}.
	\end{multline*} 
 \end{proposition}

\begin{proof}
	This is  \cite{BK1} Proposition 3.8, and follows essentially from the interpolation property \eqref{eq: interpolate}  
	of $\mu_{z_0,w_0}$.  The crucial step in the proof is the explicit calculation
	of the restriction of the measure on $\bbZ_p^2$ to $(\bbZ_p^\times)^2$.  This is done 
	by calculating the translations by $p$-th roots of unity of the power series $\wh\Theta^{*\iota}_{z_0,w_0}(s,t)$
	used in defining the measure, and then applying a certain distribution formula.  See \cite{BK1} \S 2.4 for details.
\end{proof}
  
The measure $\mu_{z_0,w_0}$ may also be used to construct the two-variable $p$-adic $L$-function of Manin-Vishik
\cite{MV} and Katz \cite{Ka}. See \cite{BK1} \S 3.3 for details.
  
%%%%%%%%%%%%%%%%%%%%%%%%%%%%%%%%%%%%%%%%%%%%%%%%

%%%%%%%%%%%%%%%%%%%%%%%%%%%%%%%%%%%%%%%%%%%%%%%%

\end{document}